\documentclass[11pt]{amsart}

\scrollmode

\usepackage{times}
\usepackage{amsmath,graphicx}
\usepackage{latexsym}
\usepackage{amscd,amsthm,amssymb}
\usepackage[all]{xy}

\newtheorem{thm}{Theorem}
\newtheorem{prop}[thm]{Proposition}
\newtheorem{lem}[thm]{Lemma}

\theoremstyle{definition}

\newtheorem{rem}[thm]{Remark}
\newtheorem{defn}[thm]{Definition}

\title{On certain exotic 4-manifolds of Akhmedov and Park}
\author{M.~J.~D.~Hamilton}
\address{Mark Hamilton \\
               Universit\"at Stuttgart\\
               Fachbereich Mathematik\\
               Pfaffenwaldring 57\\
               70569 Stuttgart\\
               Germany}
\email{mark.hamilton@math.lmu.de}
\date{February 22, 2011}
\subjclass[2000]{57R55, 57R19, 57R17}
\keywords{fibre sum, 4-manifold, symplectic, canonical class}

\begin{document}

\begin{abstract} In an article from 2008, A.~Akhmedov and B.~D.~Park constructed irreducible symplectic 4-manifolds homeomorphic but not diffeomorphic to the manifolds $\mathbb{C}P^2\#3\overline{\mathbb{C}P^2}$ and $3\mathbb{C}P^2\#5\overline{\mathbb{C}P^2}$. These manifolds are constructed by using generalized fibre sums. In this note we describe an explicit splitting of the second (co-)homology of these manifolds adapted to their construction as fibre sums. We also calculate the canonical classes of the symplectic structures. This gives a new proof for a formula derived by A.~Akhmedov, R.~\.I.~Baykur and B.~D.~Park.
\end{abstract}

\maketitle

\tableofcontents

\section{Introduction} In their article \cite{AkP}, A.~Akhmedov and B.~D.~Park constructed simply-connected irreducible symplectic 4-manifolds $U$ and $V$ homeomorphic to $\mathbb{C}P^2\#3\overline{\mathbb{C}P^2}$ and $3\mathbb{C}P^2\#5\overline{\mathbb{C}P^2}$. In particular, being irreducible under connected sum, $U$ and $V$ are not diffeomorphic to these manifolds and hence determine exotic smooth structures. The construction of these manifolds uses the generalized fibre sum, that we will recall in Section \ref{sect Generalized fibre sums}. It turns out that, even though the building blocks in this construction have non-trivial fundamental group, the manifolds $U$ and $V$ obtained as multiple fibre sums are simply-connected.  

The second (co-)homology of a generalized fibre sum of two 4-manifolds admits a canonical decomposition, see \cite{H}. In this note we want to describe this splitting for the manifolds $U$ and $V$. This will also allow us to determine the canonical classes of the symplectic structures (a formula for the canonical class has also been derived in a different way in \cite{AkBP}). It is also quite easy to see with these methods that the manifolds $U$ and $V$ have vanishing first integral homology. We will, however, not repeat the difficult part in \cite{AkP} and show that $U$ and $V$ are simply-connected. To understand why the manifolds $U$ and $V$ have vanishing first integral homology without going through the calculation of their fundamental groups was one of the starting points for the following discussion. In general we will see that the calculations on the level of homology are fairly simple.

\section{Generalized fibre sums}\label{sect Generalized fibre sums}
In the following, we use for a topological space $Y$ the abbreviations $H_*(Y)$ and $H^*(Y)$ to denote the homology and cohomology groups of $Y$ with $\mathbb{Z}$-coefficients. The homology class of an embedded, oriented surface and the surface itself are often denoted by the same symbol. Poincar\'e duality is often suppressed, so that a class and its Poincar\'e dual are denoted by the same symbol.
 
Let $M$ and $N$ be smooth, closed, oriented 4-manifolds and $\Sigma_M$ and $\Sigma_N$ closed, oriented, embedded surfaces of the same genus $g$ in $M$ and $N$. We assume that both surfaces represent homology classes of self-intersection zero. Then their normal bundles are trivial. We choose trivializations of their normal bundles corresponding to framings of the tubular neighbourhoods as $\nu\Sigma_M=\Sigma_M\times D^2$ and $\nu\Sigma_N=\Sigma_N\times D^2$. Let $\phi:\Sigma_M\rightarrow \Sigma_N$ be an orientation preserving diffeomorphism. Then the generalized fibre sum of $M$ and $N$ along $\Sigma_M$ and $\Sigma_N$ is defined as  
\begin{equation*}
X=M'\cup_\Phi N'
\end{equation*}
where $M'$ and $N'$ denote the manifolds with boundary $\Sigma_M\times S^1$ and $\Sigma_N\times S^1$ obtained by deleting the interior of the tubular neighbourhoods of the surfaces and $\Phi$ is an orientation reversing diffeomorphism $\Phi:\partial M'\rightarrow \partial N'$ that preserves the $S^1$ fibration, reverses orientation on the circles and covers the diffeomorphism $\phi$ between the surfaces. For example in the fixed framing we can consider the diffeomorphism 
\begin{equation}\label{triv gluing}
\begin{split}
\Phi\colon \Sigma_M\times S^1&\rightarrow \Sigma_N\times S^1,\\
(x,z)&\mapsto (\phi(x),\bar{z}),
\end{split}
\end{equation}
where the bar denotes complex conjugation. In general, the smooth structure of the 4-manifold $X$ can depend on the choice of the covering diffeomorphism $\Phi$. We choose one such covering and denote the resulting manifold by 
\begin{equation*}
X=M\#_\phi N.
\end{equation*}
We can specify the diffeomorphism $\phi:\Sigma_M\rightarrow \Sigma_N$ equivalently by choosing embeddings
\begin{align*}
i_M:\Sigma_g&\rightarrow M\\
i_N:\Sigma_g&\rightarrow N
\end{align*}
that realize the surfaces $\Sigma_M$ and $\Sigma_N$ as images of a fixed closed surface $\Sigma_g$ of genus $g$. Then we choose as $\phi$ the diffeomorphism $i_N\circ (i_M^{-1}|_{\Sigma_M})$. 

If $M$ and $N$ are symplectic manifolds and $\Sigma_M$, $\Sigma_N$ symplectically embedded, then the generalized fibre sum $X$ admits a symplectic structure. See references \cite{Go} and \cite{McW} for the original construction. 

\section{Construction of the exotic $\mathbb{C}P^2\#3\overline{\mathbb{C}P^2}$}
Let $K$ be a knot in $S^3$. Denote a tubular neighbourhood of $K$ by $\nu K\cong S^1\times D^2$. Let $m$ be a fibre of the circle bundle $\partial \nu K\rightarrow K$ and use an oriented Seifert surface for $K$ to define a section $l\colon K\rightarrow\partial \nu K$. The circles $m$ and $l$ are called the {\em meridian} and the {\em longitude} of $K$. Let $M_K$ be the closed 3-manifold obtained by $0$-surgery on $K$, defined as 
\begin{equation*}
M_K=(S^3\setminus \mbox{int}\,\nu K)\cup_f (S^1\times D^2).
\end{equation*}
where the gluing diffeomorphism $f$ maps in homology $\partial D^2$ onto the longitude $l$. One can show that $M_K$ has the same integral homology as $S^2\times S^1$. The meridian $m$, which bounds the fibre in the normal bundle to $K$ in $S^3$, becomes non-zero in the homology of $M_K$ and defines a generator in $H_1(M_K)$. The longitude $l$ is null-homotopic in $M_K$ since it bounds one of the $D^2$-fibres glued in. This copy of $D^2$ determines together with the Seifert surface of $K$ a closed, oriented surface in $M_K$ which intersects $m$ once and generates $H_2(M_K)$.

Let $K$ be the left-handed trefoil knot in $S^3$. Then according to equation (1) in \cite{AkP}, the fundamental group of $M_K$ is given by
\begin{equation*}
\pi_1(M_K)=\langle a,b|aba=bab, a^{-1}b^{-2}a^{-1}b^{4}=1\rangle.
\end{equation*}
Hence in homology we have again $H_1(M_K)=\mathbb{Z}$ generated by $a=b$ which are equal, under appropriate choice of orientation, to the class of the generator $m$. 

The knot $K$ is a genus one fibred knot, hence $M_K$ is a $T^2$-bundle over $S^1$. It has a section $S^1\rightarrow M_K$ whose image is equal to $b$. Consider the 4-manifold $S^1\times M_K$. It is a $T^2$-bundle over $T^2$ with section $S=x\times b$, where $x$ denotes the $S^1$-factor. In particular, the inclusion induced map $H_1(S)\rightarrow H_1(S^1\times M_K)$ is an isomorphism. Let $F$ denote a fibre of this bundle. Then according to \cite{AkP} we can write $F=\gamma_1\times \gamma_2$ where the curves $\gamma_1$ and $\gamma_2$ are homotopic to $a^{-1}b$ and $b^{-1}aba^{-1}$. In homology, both $\gamma_1$ and $\gamma_2$ vanish, hence the inclusion induced homomorphism $H_1(F)\rightarrow H_1(S^1\times M_K)$ is the zero map. Both $S$ and $F$ are embedded tori of self-intersection zero. We have $H_2(S^1\times M_K)=\mathbb{Z}^2$, generated by $S$ and $F$. The intersection form is given in this basis by 
\begin{equation*}
Q_{S^1\times M_K}=\left(\begin{array}{cc} 0 &1 \\ 1& 0 \\ \end{array}\right).
\end{equation*}
  
By a construction of W.~P.~Thurston \cite{Th}, the manifold $S^1\times M_K$ admits a symplectic structure such that both $S$ and $F$ are symplectic submanifolds. 
\begin{lem}\label{lem K knot surgery} The canonical class of the symplectic manifold $S^1\times M_K$ is given by $K_{S^1\times M_K}=0$.
\end{lem}
\begin{proof} This follows from the adjunction formula
\begin{equation*}
2g-2=\Sigma_g^2+K\Sigma_g,
\end{equation*}
since both generators $S$ and $F$ are symplectic.
\end{proof}

Consider two copies of $S^1\times M_K$ and embeddings
\begin{align*}
i_1:T^2&\rightarrow S^1\times M_K\\
i_2:T^2&\rightarrow S^1\times M_K
\end{align*}
whose images are $S$ and $F$ and which map the standard generators of $\pi_1(T^2)$ to $\{x,b\}$ and $\{\gamma_1,\gamma_2\}$, respectively.  
\begin{defn} Let $Y_K$ denote the symplectic fibre sum $(S^1\times M_K)\#_\psi (S^1\times M_K)$, where the diffeomorphism $\psi:S\rightarrow F$ is given by $i_2\circ i_1^{-1}$. 
\end{defn}
In other words \cite{FSknot}, $Y_K$ is obtained by knot surgery with the left-handed trefoil knot on $S^1\times M_K$ along the fibre $F$. It is known that knot surgery does not change the integral homology groups and the intersection form on $H_2$ and that rim tori do not exist in the knot surgered manifold (this also follows with the methods in \cite{H}). Hence we have:
\begin{lem} $H_1(Y_K)=\mathbb{Z}^2$ and $H_2(Y_K)=\mathbb{Z}^2$. 
\end{lem}
The calculation of the first homology of $Y_K$ of course also follows from the calculation of its fundamental group in \cite{AkP}.

The generators of $H_2(Y_K)$ can be described as follows: We can consider push-offs of the surface $S$ in the first copy and $F$ in the second copy of $S^1\times M_K$ into the boundary of their tubular neighbourhoods. If we choose the gluing diffeomorphism as in equation \eqref{triv gluing}, then both push-offs get identified to a torus $T_{Y_K}$ of self-intersection zero inside $Y_K$. This is one generator of $H_2(Y_K)$. The second generator is a surface $\Sigma$ of genus $2$ and self-intersection zero in $Y_K$ obtained by sewing together a punctured fibre from the first copy and a punctured section from the second copy of $S^1\times M_K$. By the Gompf construction \cite{Go} we can assume that $\Sigma$ is symplectic. The intersection form on the generators $T_{Y_K}$ and $\Sigma$ is given by
\begin{equation*}
Q_{Y_K}=\left(\begin{array}{cc} 0 &1 \\ 1& 0 \\ \end{array}\right).
\end{equation*}
\begin{lem}\label{lem K_{Y_K}} The canonical class of the symplectic manifold $Y_K$ is given by $K_{Y_K}=2T_{Y_K}$.
\end{lem}
\begin{proof} This follows from the adjunction formula since both generators $T_{Y_K}$ and $\Sigma$ are symplectic.
\end{proof}
We can also describe the inclusion induced map $H_1(\Sigma)\rightarrow H_1(Y_K)$. Consider the following part of the Mayer-Vietoris sequence for $Y_K$:
\begin{equation*}
\ldots\rightarrow H_1(T^2\times S^1)\rightarrow H_1(S^1\times M_K\setminus \nu S)\oplus H_1(S^1\times M_K\setminus \nu F)\rightarrow H_1(Y_K)\rightarrow 0.
\end{equation*}
In $S^1\times M_K\setminus \nu S$ we have the punctured fibre and in $S^1\times M_K\setminus \nu F$ the punctured section which sew together to define the surface $\Sigma$. Since $S\cdot F=1$, both the section and the fibre represent indivisible elements in homology. This implies that the meridians to these surfaces are zero in the homology of the complements of the tubular neighbourhoods and we have isomorphisms
\begin{equation*}
H_1(S^1\times M_K\setminus \nu S)\cong H_1(S^1\times M_K\setminus \nu F)\cong H_1(S^1\times M_K).
\end{equation*}
The Mayer-Vietoris sequence reduces to
\begin{equation*}
H_1(T^2)\stackrel{i_1\oplus i_2}{\longrightarrow}H_1(S^1\times M_K)\oplus H_1(S^1\times M_K)\rightarrow H_1(Y_K)\rightarrow 0.
\end{equation*}
Hence $H_1(Y_K)$ is isomorphic to the cokernel of $i_1\oplus i_2$. The map $i_1$ on homology is an isomorphism, whereas the map $i_2$ is the zero map. It follows that the inclusion maps the generators of the punctured section to the generators of $H_1(Y_K)$ and the generators of the punctured fibre to zero. In the notation of \cite{AkP}, the group $H_1(Y_K)$ has generators $y,d$ and the inclusion maps the standard generators of $H_1(\Sigma)$ to $\{y,d,0,0\}$ in that particular order.

The manifold $Y_K$ is the first building block for $U$. The second building block is the manifold $Q=(S^1\times M_K)\#2\overline{\mathbb{C}P^2}$. In $Q$ there is a symplectic surface $\Sigma'$ of genus 2 and self-intersection zero, obtained by symplectically resolving the intersection point of a torus fibre $F$ and a torus section $S$ in $S^1\times M_K$ and then blowing up at two points.

Let $h,z$ denote the generators of $H_1(Q)=H_1(S^1\times M_K)=\mathbb{Z}^2$ corresponding to the generators $b,x$ we had previously. Then the inclusion maps the standard generators of $H_1(\Sigma')$ to $\{z,h,0,0\}$ in that particular order. Choose embeddings of a reference surface of genus $2$
\begin{align*}
i_{Y_K}:\Sigma_2&\rightarrow Y_K\\
i_Q:\Sigma_2&\rightarrow Q
\end{align*}
whose images are $\Sigma$ and $\Sigma'$ and which map the standard generators of $H_1(\Sigma_2)$ to $\{y,d,0,0\}$ and $\{0,0,z,h\}$, respectively.  
\begin{defn} Let $U$ denote the fibre sum $Y_K\#_\phi Q$, where the diffeomorphism $\phi:\Sigma\rightarrow \Sigma'$ is given by $i_Q\circ i_{Y_K}^{-1}$. 
\end{defn}

\begin{prop} Rim tori do not exist in the fibre sum $U$. The 4-manifold $U$ is a homology $\mathbb{C}P^2\#3\overline{\mathbb{C}P^2}$.
\end{prop}
\begin{proof} We use the results from \cite{H}. According to \cite[Corollary 45]{H}, $H_1(U)$ is isomorphic to the cokernel of the map $i_{Y_K}\oplus i_Q:H_1(\Sigma_2)\rightarrow H_1(Y_K)\oplus H_1(Q)$. Since this map is an isomorphism, $H_1(U)=0$. Similarly, according to \cite[Theorem 51]{H}, the subgroup of rim tori in the second homology of $U$ is isomorphic to the cokernel of the map $i_{Y_K}^*+ i_Q^*:H^1(Y_K)\oplus H^1(Q)\rightarrow H^1(\Sigma_2)$. Since this map is also an isomorphism, rim tori do not occur in the 4-manifold $U$. Finally, the formulae in \cite[Corollary 40]{H} show that $b_2^+(U)=1$ and $b_2^-(U)=3$. 
\end{proof} 
The 4-manifold $U$ is symplectic, since the surfaces $\Sigma$ and $\Sigma'$ are symplectically embedded. In \cite{AkP} the gluing diffeomorphism $\phi$ is specified on the level of fundamental groups and it is shown that $U$ is simply-connected and irreducible. Hence the manifold $U$ is an exotic $\mathbb{C}P^2\#3\overline{\mathbb{C}P^2}$.

We now describe the splitting of $H_2(U)$ adapted to the fibre sum. We decompose the second homology of the manifold $Q$ as
\begin{equation*}
H_2(Q)=\mathbb{Z}\Sigma'\oplus \mathbb{Z}B_Q\oplus P(Q),
\end{equation*}
where $B_Q$ is a surface in $Q$ with $\Sigma'\cdot B_Q=1$ and $P(Q)$ denotes the orthogonal complement of $\mathbb{Z}\Sigma'\oplus \mathbb{Z}B_Q$ with respect to the intersection form. The direct sum decomposition holds, because the intersection form is unimodular on the subgroup $\mathbb{Z}\Sigma'\oplus \mathbb{Z}B_Q$, see \cite[Lemma 1.2.12]{GS}. Similarly, we have a decomposition
\begin{equation*}
H_2(Y_K)=\mathbb{Z}\Sigma\oplus \mathbb{Z}T_{Y_K}.
\end{equation*}
In this case the subgroup $P(Y_K)$ is zero. The push-offs of the
surfaces $\Sigma$ and $\Sigma'$ determine a surface $\Sigma_U$ in $U$
of genus 2 and self-intersection $0$. The punctured surfaces $B_Q$ and
$T_{Y_K}$ sew together to define a surface $B_U$ of genus equal to the
genus of $B_Q$ plus one. The surface $B_U$ has self-intersection
$B_Q^2$ since the torus $T_{Y_K}$ has self-intersection $0$. Since rim
tori and the dual vanishing (or {\it split}) classes do not exist in $U$, \cite[Theorem 59]{H} shows that
\begin{equation*}
H_2(U)= \mathbb{Z}\Sigma_U\oplus \mathbb{Z}B_U\oplus P(Q).
\end{equation*}
The subgroup $P(Q)$ is orthogonal to the first two summands. The restriction of the intersection form to $\mathbb{Z}\Sigma_U\oplus \mathbb{Z}B_U$ is of the form
\begin{equation*}
\left(\begin{array}{cc} 0 &1 \\ 1& B_Q^2 \\ \end{array}\right)
\end{equation*}
and the intersection form on $P(Q)$ is the one induced from $Q$. Note that there is an isomorphism $H_2(U)\cong H_2(Q)$ preserving the intersection form obtained by mapping $\Sigma'$ to $\Sigma_U$, $B_Q$ to $B_U$ and the identity on $P(Q)$. 

We now determine the canonical class of $U$, which depends on the choice of the surface $B_Q$.
\begin{prop} Let $E_1,E_2$ denote the exceptional spheres in $Q$. Then the canonical classes of the symplectic 4-manifolds $Q$ and $U$ are given by
\begin{equation*}
K_Q=E_1+E_2
\end{equation*}
and 
\begin{equation*}
K_U=(2+K_QB_Q-2B_Q^2)\Sigma_U+2B_U+(K_Q-2B_Q-(K_QB_Q-2B_Q^2)\Sigma').
\end{equation*}
In the formula for $K_U$ the term $K_Q-2B_Q-(K_QB_Q-2B_Q^2)\Sigma'$ is an element of $P(Q)$.
\end{prop} 
\begin{proof} The formula for $K_Q$ follows from Lemma \ref{lem K knot surgery} and the adjunction formula for the exceptional spheres. According to \cite[Theorem 89]{H}, the canonical class of $U$ is given by
\begin{equation*}
K_U=\overline{K_Q}+b_UB_U+(\eta_U+\eta'_U)\Sigma_U,
\end{equation*}
where
\begin{align*}
\overline{K_Q}&=K_Q-(2g-2)B_Q-(K_QB_Q-(2g-2)B_Q^2)\Sigma'\in P(Q)\\
b_U&=2g-2\\
\eta_U&=K_{Y_K}T_{Y_K}+1-(2g-2)T_{Y_K}^2\\
\eta'_U&=K_QB_Q+1-(2g-2)B_Q^2.
\end{align*}
In our case, $g=2$ and $K_{Y_K}$ is given by Lemma \ref{lem K_{Y_K}}.
\end{proof}
For instance, we can choose as $B_Q$ the section $S$ or the fibre $F$ in $S^1\times M_K$. In both cases $B_Q^2=0$ and $K_QB_Q=0$, hence
\begin{equation*}
K_U=2\Sigma_U+2B_U+(K_Q-2B_Q).
\end{equation*}
With the formula for the intersection form of $U$ it follows that $K_U^2=6$, as expected from the formula $K_U^2=2e(U)+3\sigma(U)$.

\section{Construction of the exotic $3\mathbb{C}P^2\#5\overline{\mathbb{C}P^2}$}

The first building block for the exotic 4-manifold $V$ is $R=T^4\#2\overline{\mathbb{C}P^2}$. Fix a factorization $T^4=T^2\times T^2$ and choose a symplectic structure on $T^4$ such that both tori are symplectically embedded. Symplectically resolving the intersection point of the two tori and blowing up twice we obtain a symplectic surface $\Sigma''$ of genus 2 and self-intersection zero in $R$. 

Let $\alpha_i$, $i=1,\ldots,4$, denote the generator of $H_1(R)=H_1(T^4)$ corresponding to $i$-th circle factor. Then the inclusion maps the standard generators of $H_1(\Sigma'')$ to $\{\alpha_1,\alpha_2,\alpha_3,\alpha_4\}$ in that particular order. In particular, the inclusion induced map $H_1(\Sigma'')\rightarrow H_1(R)$ is an isomorphism.  

To describe the second building block of the manifold $V$, we consider two copies of the manifold $Y_K$ constructed above. Recall that in $Y_K$ there is a symplectic surface $\Sigma$ of genus 2 and self-intersection zero. If $y,d$ denote the generators of $H_1(Y_K)$ then the inclusion maps the standard generators of $H_1(\Sigma)$ to $\{y,d,0,0\}$. Let $t,s$ denote generators of the second copy of $Y_K$ corresponding to $y,d$. Choose embeddings of a reference surface of genus $2$
\begin{align*}
i_{Y_{K1}}:\Sigma_2&\rightarrow Y_K\\
i_{Y_{K2}}:\Sigma_2&\rightarrow Y_K
\end{align*}
whose images are the surfaces $\Sigma$ in the first and second copy of $Y_K$ and which map the standard generators of $H_1(\Sigma_2)$ to $\{y,d,0,0\}$ and $\{0,0,t,s\}$, respectively.  
\begin{defn} Let $X_K$ denote the symplectic fibre sum $Y_K\#_\psi Y_K$, where the diffeomorphism $\psi:\Sigma\rightarrow \Sigma$ is given by $i_{Y_{K2}}\circ i_{Y_{K1}}^{-1}$. 
\end{defn}
\begin{lem} Rim tori do not exist in the fibre sum $X_K$. We have $H_1(X_K)=0$ and $H_2(X_K)=\mathbb{Z}^2$.
\end{lem}
\begin{proof} By construction, the map $i_{Y_{K1}}\oplus i_{Y_{K2}}:H_1(\Sigma_2)\rightarrow H_1(Y_K)\oplus H_1(Y_K)$ is an isomorphism. Hence $H_1(X_K)$, which is isomorphic to the cokernel of this map, vanishes. Similarly, $i_{Y_{K1}}^*+ i_{Y_{K2}}^*:H^1(Y_K)\oplus H^1(Y_K)\rightarrow H^1(\Sigma_2)$ is an isomorphism. Therefore, rim tori do not exist in the fibre sum $X_K$. Finally, the claim that $b_2(X_K)=2$ follows from \cite[Corollary 40]{H}.
\end{proof}
This lemma has also been proved in \cite{Ak}. We can describe the splitting of the second homology of $X_K$ adapted to the fibre sum as follows. The second homology of the first copy of $Y_K$ splits as
\begin{equation*}
H_2(Y_K)=\mathbb{Z}\Sigma\oplus \mathbb{Z}T_{Y_K}.
\end{equation*}
and similarly for the second copy. The push-offs of the surfaces $\Sigma$ in the first and second copy determine a symplectic surface $\Sigma_{X_K}$ of genus 2 and self-intersection zero in $X_K$. The punctured tori $T_{Y_K}$ in the first and second copy of $Y_K$ sew together to determine a surface $B_{X_K}$ of genus 2 and self-intersection zero in $X_K$. By the Gompf construction we can assume that $B_{X_K}$ is symplectic. Since rim tori and vanishing classes do not exist in $Y_K$, we have
\begin{equation*}
H_2(X_K)=\mathbb{Z}\Sigma_{X_K}\oplus \mathbb{Z}B_{X_K}.
\end{equation*}
The intersection form in this basis is given by 
\begin{equation*}
Q_{X_K}=\left(\begin{array}{cc} 0 &1 \\ 1& 0 \\ \end{array}\right)
\end{equation*}
hence $X_K$ is a homology $S^2\times S^2$.
\begin{lem}\label{lem K_{X_K}} The canonical class of the symplectic manifold $X_K$ is given by $K_{X_K}=2\Sigma_{X_K}+2B_{X_K}$.
\end{lem}
\begin{proof} This follows from the adjunction formula since both surfaces $\Sigma_{X_K}$ and $B_{X_K}$ are symplectic.
\end{proof}

We choose embeddings of a reference surface of genus $2$
\begin{align*}
i_{R}:\Sigma_2&\rightarrow R\\
i_{X_K}:\Sigma_2&\rightarrow X_K
\end{align*}
whose images are the surfaces $\Sigma''$ and $\Sigma_{X_K}$ and which map the standard generators of $H_1(\Sigma_2)$ to $\{\alpha_1,\alpha_2,\alpha_3,\alpha_4\}$ and $\{0,0,0,0\}$, respectively.  
\begin{defn} Let $V$ denote the symplectic fibre sum $R\#_\phi X_K$, where the diffeomorphism $\phi:\Sigma''\rightarrow \Sigma_{X_K}$ is given by $i_{X_K}\circ i_{R}^{-1}$. 
\end{defn}
\begin{prop} Rim tori do not exist in the fibre sum $V$. The 4-manifold $V$ is a homology $3\mathbb{C}P^2\#5\overline{\mathbb{C}P^2}$.
\end{prop}
\begin{proof} By construction, the map $i_{R}\oplus i_{X_K}:H_1(\Sigma_2)\rightarrow H_1(R)\oplus H_1(X_K)=H_1(R)$ is an isomorphism. Hence $H_1(V)$, which is isomorphic to the cokernel of this map, vanishes. Similarly, $i_{R}^*+ i_{X_K}^*:H^1(R)\oplus H^1(X_K)=H^1(R)\rightarrow H^1(\Sigma_2)$ is an isomorphism. Therefore, rim tori do not exist in the fibre sum $V$. Finally, the claim that $b_2^+(V)=3$ and $b_2^-(V)=5$ follows again from \cite[Corollary 40]{H}.
\end{proof}
In \cite{AkP} it is shown that $V$ is simply-connected and irreducible. Hence the manifold $V$ is an exotic $3\mathbb{C}P^2\#5\overline{\mathbb{C}P^2}$.

We describe the splitting of $H_2(V)$ adapted to the fibre sum. We first decompose the second homology of the manifold $R$ as
\begin{equation*}
H_2(R)=\mathbb{Z}\Sigma''\oplus \mathbb{Z}B_R\oplus P(R),
\end{equation*}
where $B_R$ is a surface in $R$ with $\Sigma''\cdot B_R=1$ and $P(R)$ denotes the orthogonal complement of $\mathbb{Z}\Sigma''\oplus \mathbb{Z}B_R$ with respect to the intersection form. We also have a decomposition
\begin{equation*}
H_2(X_K)=\mathbb{Z}\Sigma_{X_K}\oplus \mathbb{Z}B_{X_K},
\end{equation*}
where both $\Sigma_{X_K}$ and $B_{X_K}$ are surfaces of genus 2 and self-intersection zero. The push-offs of the surfaces $\Sigma''$ and $\Sigma_{X_K}$ determine a surface $\Sigma_V$ in $V$ of genus 2 and self-intersection $0$. The punctured surfaces $B_R$ and $B_{X_K}$ sew together to define a surface $B_V$ of genus equal to the genus of $B_R$ plus two. The surface $B_V$ has self-intersection $B_R^2$. Since rim tori and the dual vanishing classes do not exist in $V$, \cite[Theorem 59]{H} shows that
\begin{equation*}
H_2(V)= \mathbb{Z}\Sigma_V\oplus \mathbb{Z}B_V\oplus P(R).
\end{equation*}
The subgroup $P(R)$ is orthogonal to the first two summands. The restriction of the intersection form to $\mathbb{Z}\Sigma_V\oplus \mathbb{Z}B_V$ is of the form
\begin{equation*}
\left(\begin{array}{cc} 0 &1 \\ 1& B_R^2 \\ \end{array}\right)
\end{equation*}
and the intersection form on $P(R)$ is the one induced from $R$. There is again an isomorphism $H_2(V)\cong H_2(R)$ preserving the intersection form. 

We determine the canonical class of $V$, depending on the choice of the surface $B_R$. 
\begin{prop} Let $E_1,E_2$ denote the exceptional spheres in $R$. Then the canonical classes of the symplectic 4-manifolds $R$ and $V$ are given by
\begin{equation*}
K_R=E_1+E_2
\end{equation*}
and
\begin{equation*}
K_V=(4+K_RB_R-2B_R^2)\Sigma_V+2B_V+(K_R-2B_R-(K_RB_R-2B_R^2)\Sigma'').
\end{equation*}
In the formula for $K_V$ the term $K_R-2B_R-(K_RB_R-2B_R^2)\Sigma''$ is an element of $P(R)$.
\end{prop}
\begin{proof} The first claim follows because $K_{T^4}=0$. According to \cite[Theorem 89]{H}, the canonical class of $V$ is given by
\begin{equation*}
K_V=\overline{K_R}+b_VB_V+(\eta_V+\eta'_V)\Sigma_V,
\end{equation*}
where
\begin{align*}
\overline{K_R}&=K_R-(2g-2)B_R-(K_RB_R-(2g-2)B_R^2)\Sigma''\in P(R)\\
b_V&=2g-2\\
\eta_V&=K_RB_R+1-(2g-2)B_R^2\\
\eta'_V&=K_{X_K}B_{X_K}+1-(2g-2)B_{X_K}^2.
\end{align*}
In our case, $g=2$ and $K_{X_K}$ is given by Lemma \ref{lem K_{X_K}}. 
\end{proof}
For example, we can take as $B_R$ one of the torus factors in $T^4=T^2\times T^2$. Then $B_R^2=0$ and $K_RB_R=0$, hence
\begin{equation*}
K_V=4\Sigma_V+2B_V+(K_R-2B_R).
\end{equation*}
With the formula for the intersection form we have $K_V^2=14$, as expected.

\begin{rem} In \cite{Ak2}, A.~Akhmedov constructed irreducible
  symplectic 4-manifolds $Y$ and $X$ homeomorphic to
  $\mathbb{C}P^2\#5\overline{\mathbb{C}P^2}$ and
  $3\mathbb{C}P^2\#7\overline{\mathbb{C}P^2}$ using generalized fibre
  sums. The building blocks of $X$ and $Y$ are the manifolds $X_K$ and
  $Y_K$ and $Z=T^2\times S^2\#4\overline{\mathbb{C}P^2}$. The manifold
  $Z$ admits a Lefschetz fibration with fibres of genus $2$. Let
  $\Sigma_2'$ denote a regular fibre and $a_1,b_1$ the generators of
  $H_1(Z)$ in the notation of \cite{Ak2}. Then the inclusion induced
  homomorphism maps the standard generators of $H_1(\Sigma_2')$ to
  $\{a_1,b_1,-a_1,-b_1\}$ in that particular order.

The manifold $Y$ is obtained as a generalized fibre sum of $Y_K$ and
$Z$. Using similar arguments as before one can show that rim tori do
not exist in the fibre sum $Y$ and calculate the canonical class. The
manifold $X$, homeomorphic to
$3\mathbb{C}P^2\#7\overline{\mathbb{C}P^2}$ and obtained as a fibre
sum of $X_K$ and $Z$, however, does contain rim
tori. The subgroup of rim tori in the second homology of $X$ is given
by the cokernel of the inclusion induced homomorphism
$H^1(Z)\rightarrow H^1(\Sigma_2)$ and hence is isomorphic to
$\mathbb{Z}^2$. There also exists a dual subgroup of vanishing
classes. If $R(X)$ and $S'(X)$ denote the groups of rim tori and
vanishing classes, then in a similar way as before 
\begin{equation*}
H_2(X)=\mathbb{Z}\Sigma_X\oplus\mathbb{Z}B_X\oplus P(Z)\oplus
R(X)\oplus S'(X).
\end{equation*}
In this case the canonical class of $X$ contains a rim tori
contribution that depends on the choice of covering diffeomorphism
$\Phi$ used in the construction of the fibre sum. See \cite[Theorem
89]{H} for the general formula.  

\end{rem}

\bibliographystyle{amsplain}

\bigskip
\bigskip

\end{document}